\documentclass[12pt]{article}
\usepackage{fullpage}
\usepackage{setspace}
\usepackage{amsmath}
\usepackage{amssymb}
\usepackage{amsfonts}
\usepackage{amsthm}
\usepackage{newlfont}
\usepackage{epsfig}
\usepackage{color}
\usepackage{authblk}
\theoremstyle{plain}
\newtheorem{thm}{Theorem}[section]

\newtheorem{prop}{Proposition}[section]

\theoremstyle{definition}
\newtheorem{defn}{Definition}[section]
\newtheorem{ex}{Example}[section]
\newtheorem{proc}{Procedure}[section]

\theoremstyle{remark}
\newtheorem{rem}{Remark}[section]

\title{Adjusting for selection bias in testing multiple families of hypotheses}
\author{\textsc{Yoav Benjamini}\\\textit{Department of Statistics and Operations Research, The Sackler School of Exact Sciences, Tel Aviv University, Tel Aviv, Israel} \\ybenja@post.tau.ac.il\and \textsc{Marina Bogomolov}\\\emph{Department of Statistics and Operations Research, The Sackler School of Exact Sciences, Tel Aviv University, Tel Aviv, Israel}\\marinazh@post.tau.ac.il}
\begin{document}
\maketitle 
\abstract{In many large multiple testing problems the hypotheses are
divided into families. Given the data, families with evidence for
true discoveries are selected, and hypotheses within them are
tested. Neither controlling the error-rate in each family separately
nor controlling the error-rate over all hypotheses together can
assure that an error-rate is controlled in the selected families. We
formulate this concern about selective inference in its generality,
for a very wide class of error-rates and for any selection
criterion, and present an adjustment of the testing level inside the
selected families that retains the average error-rate over the
selected families.}
\newpage
\section{Introduction}
In modern statistical challenges one is often presented with a
(possibly) large set of large families of hypotheses. In fMRI
research interest lies with the locations (voxels) of activation
while a subject is involved in a certain cognitive task. The brain
is divided into regions (either anatomic or functional), and the
hypotheses regarding the locations in each region define a family
(see, e.g., Benjamini and Heller, 2007, Pacifico et al., 2004).
 Searching for differentially expressed genes,
 the genes are often divided into gene sets, defined by prior biological
 knowledge. Each gene set defines a family of hypotheses (see Subramanian et al., 2005,
 Heller et al., 2009).
In the above examples, the families are clusters of units of
interest: voxels or genes. Another problem having similar structure
can be identified in multi-factor analysis of variance (ANOVA),
where for each factor interest lies with the family of pairwise
comparisons between the levels of that factor. Sometimes the set of
hypotheses has a complex structure and can be divided into families
in different ways. An example of such research is the voxelwise
genome-wide association study, see Stein et al. (2010). In this
study the relation between 448,293 Single Nucleotyde Polymorphisms
(SNPs) and volume change in 31,622 voxels (total of
$448,293\cdot31,622$ hypotheses) is explored across 740 elderly
subjects. We may view this problem as a family for each gene,
 or a family for each voxel. This example is considered in detail in Section \ref{realdataex}.


 Since in many of these cases the families are large, and we are in
search for the few interesting significant findings in each family,
it is essential to control for multiplicity within each family.
Efron (2008) showed that controlling for multiplicity globally on
the pooled set of hypotheses (ignoring the division of this set into
families) may distort the inferences inside the families in both
directions - the true discoveries may remain undiscovered in
families rich in potential discoveries and too many false ones may
be discovered in families rich in true null hypotheses.

Many error criteria are in use to deal with the possible inflation
of false discoveries when facing multiplicity. In this work we
address all error-rates that can be written as
 $E(\mathcal{C})$ for some (random) measure of the errors performed
$\mathcal{C}$. These include the \textit {per-family error rate}
(PFER), where $\mathcal{C}=V$ the number of type I errors made; the
\textit{family-wise error rate} (FWER), where $\mathcal{C} =
\textbf{I} _{\{V \geq 1\}}$; the \textit{false discovery rate} (FDR)
where $\mathcal{C} =FDP$ the proportion of false discoveries
(introduced in Benjamini and Hochberg, 1995); the false-discovery
exceedance, FDX, i.e.
$\textmd{Pr}(FDP>\gamma)=E(\textbf{I}_{\{FDP>\gamma\}})$ for some
pre-specified $\gamma$ (see van der Laan et al. (2004a) and Genovese
and Wasserman, 2006); and the generalized error rates, $k$-FWER,
i.e. $\textmd{Pr}(V\geq k)=E(\textbf{I}_{\{V\geq k\}})$ (see van der
Laan et al., 2004, Lehmann and Romano, 2005) and $k$-FDR$=E(FDP\cdot
\textbf{I}_{\{V\geq k\}})$ (introduced in Sarkar, 2007).  However,
there are error-rates which cannot be written in the form
$E(\mathcal{C})$ for some random variable $\mathcal{C}$. For
example, the Bayesian FDR (Fdr), proposed by Efron and Tibshirani
(2002) is $\frac{E(V)}{E(R)}$, and is not of the form
$E(\mathcal{C})$. The positive false discovery rate proposed by
Storey (2003) is
$E(\frac{V}{R}|R>0)=\frac{E(FDP)}{\textmd{Pr}(R>0)}$ and also cannot
be written as $E(\mathcal{C})$. See Farcomeni (2008) for a good
review of multiple error criteria, the relationship between them and
different multiple testing procedures.

%
Suppose we control in each family of hypotheses separately a
criterion at level $q$. Let $\mathcal{C}_i$ be the random value of
$\mathcal{C}$ measuring the errors performed in family $i$,
$i=1,\ldots,m$, so $E(\mathcal{C}_i)\leq q$.
It is trivial that we
also control for $E(\mathcal{C})$ on the average over \textit{all}
families as well, i.e.
\begin{align}\label{av}
E\left(\frac{\sum_{i=1}^m\mathcal{C}_i}{m}\right)\leq q.
\end{align}

In many cases investigators tend to select promising families first,
based on the data at hand, and then look for significant findings
only within the selected families. In these cases it is common to
control for $E(\mathcal{C})$ in each selected family separately. Let
us consider the widely used ANOVA with two or more factors as an
example. The researcher first selects the significant factors, and
then performs post-hoc tests (pairwise comparisons) within each
selected factor. Usually the FWER is controlled within each selected
family of pairwise comparisons using Tukey's procedure, but in large
problems the FDR has also been suggested for that purpose (see
Williams et al., 1999).

When considering only the selected families, we might wish to
control the expected value of $\mathcal{C}_i$ in each selected
family. Unfortunately, the goal of such conditional control for any
combination of selection rule and testing procedure and for any
configuration of true null hypotheses is impossible to achieve, as
shown in the following example.
\begin{ex}
Suppose one of the families, say family $i$, consists only of true
null hypotheses. Let $E(\mathcal{C}^*)$ be an error-rate which
reduces to FWER when all the null hypotheses are true (e.g. FWER,
FDR, FDX). Let Proc* be an $E(\mathcal{C}^*)$-controlling procedure
applied in each selected family at some level $q'$. Consider the
following selection rule $\mathcal{S}$: select the families where at
least one rejection is made by Proc*. It is obvious that if family
$i$ is selected, there is at least one rejection in it, therefore
$\mathcal{C}_i=\textbf{I}_{\{V_i>0\}}=1$. Hence, we have shown that
for any $E(\mathcal{C})$-controlling procedure used for testing the
selected families, we can find a selection rule $\mathcal{S}$ such
that $E(\mathcal{C}
_i^*\,|\,\,i\,\,is\,\,selected\,\,by\,\,\mathcal{S})=1$.
\end{ex}

 Therefore, we would like to achieve
a more modest goal: the control of the expected average value of
$\mathcal{C}$ over the \textit{selected} families, where the average
is 0 if no family is selected.

Formally, let ${\textbf{P}_i}$ be the set of p-values belonging to
family $i$, $i=1,\ldots, m$. Let $\textbf{P}$ be the ensemble of
these sets: $\textbf{P}=\{\textbf{P}_i\}_{i=1}^m$. Let $\mathcal{S}$
be a selection procedure using as input the p-values $\mathbf{P}$,
identifying the indices of the selected families. Note that if the
selection does not depend on $\mathbf{P}$, the situation is similar
to (1), so assuming $\mathcal{S}=\mathcal{S}(\textbf{P})$ is not a
restriction. Define $|\mathcal{S}(\textbf{P})|$, the number of
selected families. The error criterion that interests us is
\begin{align}
E(\mathcal{C}_{\mathcal{S}})=E\Bigg(\frac{\sum_{i\in
\mathcal{S}(\textbf{P})}\mathcal{C}_i}{\max(|\mathcal{S}(\textbf{P})|,
1)}\Bigg)\label{mainmeas}
\end{align}
Let us first illustrate how the choice of $\mathcal{C}$ is reflected
in the resulting error-rate. When
$\mathcal{C}=\textbf{I}_{\{V>0\}}$, this error measure is the
expected proportion of families with at least one type I error out
of all the selected families. In this case it is similar to OFDR
defined in Heller et al. (2009) in the framework of microarray
analysis. When $\mathcal{C}=FDP$, the error measure in
(\ref{mainmeas}) becomes less stringent: it is the expected average
$FDP$ over the selected families. The difference between the average
$FDP$ and the proportion of families with at least one type I error
may be very large. If three families are selected, with false
discovery proportions equal to 0.04, 0.05 and 0.06 respectively, the
average $FDP$ is 0.05, whereas the proportion of families with at
least one type I error is 1. The choice between these two
error-rates should be guided by the application. If one can bear
some false discoveries in the selected families as long as the
average $FDP$ over the selected families is small, the control of
the expected average $FDP$ may suffice. Alternatively, if one wishes
to avoid even one false discovery in a selected family, control of
the expected proportion of families with at least one type I error
would be appropriate.

Now we would like to illustrate the difference between controlling
the expected average FDP over the selected families and controlling
the FDR globally for the combined set of discoveries. Assume 40
families of hypotheses are selected. There are 36 families with one
rejection in each, and there are no false discoveries in these
families. In each of the remaining 4 families there are 10
rejections, 5 out of which are false discoveries. Thus in 36
selected families $FDP=0$, while in the remaining 4 families
$FDP=0.5$. The average $FDP$ over the selected families is
$\frac{4\times0.5}{40}=0.05.$ The total number of discoveries is 76,
20 out of which are false discoveries. Therefore the $FDP$ for the
combined set of discoveries is $\frac{20}{76}=0.26.$ This simple
example suggests that control of the expected average $FDP$ over the
selected families does not imply control of the FDR for the combined
set of discoveries. Let us now illustrate that controlling the FDR
for the combined set of discoveries need not guarantee control of
the expected average FDP over the selected families. Assume that
there are 20 selected families with one erroneous and one correct
rejection, whereas in each of the other 20 families there are 18
rejections, all of which are correct. The total number of
discoveries is 400, 20 out of which are false discoveries.
Therefore, the $FDP$ for the combined set of discoveries is
$\frac{20}{400}=0.05$. However, the average $FDP$ over the selected
families is $\frac{20\times0.5}{40}=0.25$.

Controlling an error-rate on the average over the selected families
may have important advantages versus controlling it globally for the
combined set of discoveries. It does give some level of confidence
in the discoveries within each selected family - even if only on the
average. In many applications controlling an error-rate on the
average over the selected families is simply a more appropriate
measure of error for the interpretation of the results than
controlling an error-rate globally for the combined set of
discoveries. This important point is addressed in Section 3, where
we discuss the structure of the families in view of the relevance of
the control on the average over the selected, and illustrated with
an application in Section 4. Even in the problems where no selection
takes place, Efron (2008) argues that that one should obtain control
of an error-rate in each family separately, implying control on the
average (over all the families). When the selection takes place,
control on the average (over the selected families) becomes even
more important. Finally, in some cases power may be gained by
controlling an error-rate on the average rather than globally for
the combined set of discoveries, even though this is not the
motivating reason for our emphasizing the control on the average
over the selected.

Control on the average over the selected is a manifestation of
selective inference ideas developed in Benjamini and Yekutieli
(2005). In that paper the authors made an important distinction
between simultaneous and selective goals in inference on multiple
parameters, in the context of confidence intervals (CIs) for the
selected parameters. Simultaneous inference is relevant when the
control of the probability that at least one CI does not cover its
parameter is needed. As a result, the simultaneous control also
holds for any selected subset. However, when CIs are built only for
one set of selected parameters, the goal need not be that strict,
and the authors suggest a more liberal property: the control of the
expected proportion of parameters not covered by their CIs among the
selected parameters, where the proportion is 0 if no parameter is
selected (FCR). Setting the goal of selective inference for the
testing of multiple families and adopting the error measure in
(\ref{mainmeas}) is analogous to FCR in the framework of building
CIs for the selected parameters.

Note that $E(\mathcal{C}_\mathcal{S})$ reduces to (\ref{av}), when
$\mathcal{S}(\textbf{P})\equiv\{1,\ldots,m\}$, i.e. when there is
practically no selection. Unfortunately, when the error measure is
averaged over the selected families, the expectation is not
controlled, as demonstrated in the following example.
\begin{ex}\label{ex1}
A family of $n$ hypotheses, for each of which we have a $p$-value,
is selected if the minimum $p$-value in it is less than $0.05$. Each
selected family is tested using Bonferroni procedure at level
$\alpha=0.05$. 
Further assume we have $m$ such families, and all the null
hypotheses are true (with uniformly distributed $p$-values).
Obviously, the expected value of averaged $\textbf{I}_{\{V \geq
1\}}$ is also controlled when the average is taken over all
families. Let us demonstrate what happens to the average
$\mathcal{C}=\textbf{I}_{\{V \geq 1\}}$ (average FWER hereafter)
when taken only over the selected families, namely to
$E(\mathcal{C}_{\mathcal{S}})$, over various values of $m$ and $n$.
In this case it can be explicitly determined (see Appendix 1), and
is given in the following table.
\begin{table}[h]
\begin{center}
\begin{tabular}{cccc}
  $m$ & $n$ & $E(|\mathcal{S}(\textbf{P})|/m)$ &
  $E(\mathcal{C}_{\mathcal{S}})$\\
  20& 100 & 0.99 & 0.049 \\
  100 & 20 & 0.64 & 0.076 \\
  100 & 10 & 0.40 & 0.122 \\
  100 & 2 & 0.1 & 0.506 \\
\end{tabular}
\caption{Illustration of the selection bias in Example \ref{ex1}.
There are $m$ families with $n$ hypotheses in each. All  hypotheses
are null. $\mathcal{S}(\textbf{P})$ is the selected set of families,
containing all families the minimum p-value in which is less than
$0.05$. Each selected family is tested using the Bonferroni
procedure at level $0.05$, assuring that
$E(\mathcal{C}_i)=E(\textbf{I}_{\{V_i\geq 1\}})\leq 0.05$.
$\mathcal{C}_{\mathcal{S}}$ is the average number of families where
at least one type I error was made. It can be seen that as the
selection becomes more stringent, the selection bias is more
severe.}
\end{center}
\end{table}

\noindent One can immediately observe from the last column that in
this example the average FWER over the selected families can climb
high and reach above $0.5$, while with no selection the level should
be $0.05$. It is also clear that the average FWER over the selected
increases when the extent of selection (presented in the third
column) becomes more extreme. Similar results were observed for
average PFER ($E(\mathcal{C}_\mathcal{S})$ for $\mathcal{C}=V$)
rather than average FWER. In this particular example the extent of
selection does not depend on the number of families $m$, but only on
$n$, but this need not be the case for other selection rules.
\end{ex}

The main result of this paper is that in order to assure the control
of $E(\mathcal{C}_{\mathcal{S}})$, we should control for
$E(\mathcal{C}_i)$ in each selected family $i$ at a more stringent
level: the nominal level $q$ should be multiplied by the proportion
of the selected families among all the families. This result, under
some limiting conditions, is the focus of Theorem 2.1. A general
result of the same nature, covering more complicated selection
rules, such as multiple comparisons procedures that make use of
plug-in estimators, is given in Theorem 2.2.


\section{Selection adjusted testing of families}
When all the families are selected with probability 1, no adjustment
to the testing levels should be done because the average over the
selected families is the average over all. As the selection rule is
more stringent and tends to select less families, the adjustment
should be more severe. For clarity of exposition and enhancing
intuition, we first introduce the adjustment for simple selection
rules, first introduced by Benjamini and Yekutieli (2005) in the
context of parameter selection, and only then turn to the general
case.
\begin{defn}\label{defsimple}(Simple selection rule) A selection
rule is called simple if for each selected family $i$, when the
p-values not belonging to family $i$ are fixed and the p-values
inside family $i$ can change as long as family $i$ is selected, the
number of selected families remains unchanged.
\end{defn}
It is easy to see that many selection rules are indeed simple in the
above sense. Any rule where a family is selected based only on its
own p-values is a simple selection rule, as in Example 1.1.
In Section \ref{hyptest} we show, that when the selection of the
families is done using hypothesis testing, the widely used step-up
and step-down multiple testing procedures provide simple selection
rules, even though the decision whether a family is selected or not
depends on the p-values belonging to other families as well.
However, not all the selection rules are simple: examples are
adaptive multiple testing procedures, as noted in Section
\ref{hyptest}.

The following procedure offers the selection adjustment when the
families are selected using a simple selection rule.
\begin{proc}[Simple Selection-Adjusted Procedure]\label{mainproc1} $\,$
\smallskip

     \noindent 1. Apply the selection rule $\mathcal{S}$ to the ensemble of sets
$\textbf{P}$, identifying the selected set of families
$\mathcal{S}(\textbf{P})$. Let $R$ be the number of selected
families (i.e. $R=|\mathcal{S}(\textbf{P})|$). \\2. Apply
$E(\mathcal{C})$-controlling procedure in each selected family
separately at level
\begin{align*}\frac{Rq}{m}.\end{align*}
\end{proc}
\begin{thm}\label{mainthm} If the p-values across the families are independent, then for any
simple selection rule $\mathcal{S}(\textbf{P})$, for any error-rate
$E(\mathcal{C})$ such that $\mathcal{C}$ takes values in a countable
set, and for any $E(\mathcal{C})$-controlling procedure valid for
the dependency structure inside each family, the Simple
Selection-Adjusted Procedure guarantees
$E(\mathcal{C}_{\mathcal{S}})\leq q$.
\end{thm}
\begin{rem} For all error-rates known to us, $\mathcal{C}$ is
a count or a ratio of counts, so the condition on the values
$\mathcal{C}$ takes is satisfied.\end{rem}
\begin{proof}[Proof of Theorem \ref{mainthm}]
The idea of the proof is similar to the proof of Theorem 1 in
Benjamini and Yekutieli (2005). For each error criterion
$E(\mathcal{C})$, let $\mathcal{C}_+$ be the countable support of
$\mathcal{C}$. Since the selection rule is simple, we can define the
following event on the space of all the p-values not belonging to
family $i$: if family $i$ is selected , $k$ families are selected
including family $i$. Denote this event by $C_k^{(i)}$. For any
simple selection rule
\begin{align}
E(\mathcal{C}_{\mathcal{S}})=\sum_{i=1}^m\sum_{k=1}^m\frac{1}{k}\sum_{c\in\mathcal{C}_+}c\textmd{Pr}(\mathcal{C}_i=c,
\, i\in \textit{S}(\textbf{P}), C_k^{(i)})\label{forappb}
\end{align}
 Note that Simple Selection-Adjusted
Procedure does not reject any hypothesis in families which are not
selected. Therefore $\mathcal{C}_i=0$ for each family $i$ that is
not selected. Hence, for this procedure we obtain
\begin{align}
E(\mathcal{C}_{\mathcal{S}})&=\sum_{i=1}^m\sum_{k=1}^m\frac{1}{k}\sum_{c\in\mathcal{C}_+}c\textmd{Pr}(\mathcal{C}_i=c,
\,
C_k^{(i)})\notag\\&=\sum_{i=1}^m\sum_{k=1}^m\frac{1}{k}\sum_{c\in\mathcal{C}_+}c\textmd{Pr}(\mathcal{C}_i=c
)
\textmd{Pr}(C_k^{(i)})\label{c1}\\&=\sum_{i=1}^m\sum_{k=1}^m\frac{1}{k}E(\mathcal{C}_i
)
\textmd{Pr}(C_k^{(i)})\label{imp1}
\end{align}
%
%
%
%
%
Equality (\ref{c1}) follows from the independence between $P_i$ and
the set of p-values not belonging to family $i$, for each
$i=1,\ldots,m$.  In expression (\ref{imp1}), for each $k$ and $i$,
$\mathcal{C}_i$ is the value of random variable $\mathcal{C}$ in
family $i$, when a valid $E(\mathcal{C})$-controlling procedure is
applied at level $\frac{kq}{m}$ in each selected family. 
Since there are no rejections in families that are not selected,
$\mathcal{C}_i$ takes the value 0 there, so $E(\mathcal{C}_i)\leq
\frac{kq}{m}$ for each $i=1,\ldots,m$.
Now, using this inequality and the fact that
$\sum_{k=1}^m\textmd{Pr}(C_k^{(i)})=1$ for each $i=1,\ldots,m$, we
obtain
\begin{align}
\sum_{i=1}^m\sum_{k=1}^m\frac{1}{k}E(\mathcal{C}_i)
\textmd{Pr}(C_k^{(i)})\leq\sum_{i=1}^m\sum_{k=1}^m\frac{1}{k}\cdot\frac{kq}{m}\cdot\textmd{Pr}(C_k^{(i)})=q\label{c2}
\end{align}
Results (\ref{imp1}) and (\ref{c2}) complete the proof.
\end{proof}

Theorem \ref{mainthm} supplies the adjustment of the testing level
in each selected family which is sufficient for the control of
$E(\mathcal{C}_S)$ when the selection rule is simple. We will now
show that in some special cases this adjustment is necessary,
adopting Example 6 in Benjamini and Yekutieli (2005) for our needs.
\begin{ex}
Assume all the families are of equal size, $n$. All the hypotheses
are null, all the p-values are jointly independent and uniformly
distributed. Let us order the families by their minimal p-values.
The simple selection rule is to choose the $k$ families with the
smallest minimal p-values. Assume that each selected family is
tested using the Bonferroni procedure at level $q'$. In this case
the average error-rate over the
 selected is
\begin{align*}
E\left(\mathcal{C}_\mathcal{S}\right)=\frac{E\left(\sum_{i\in
\mathcal{S}(\textbf{P})}V_i\right)}{k}
\end{align*}
\end{ex}
The families where at least one type I error is made are
the families with the smallest minimal p-values. Therefore, if 
$\sum_{i=1}^m\textbf{I}_{\{V_i\geq 1\}}\leq k$, we obtain
$$\sum_{i\in \mathcal{S}(\textbf{P})}V_i=\sum_{i=1}^mV_i$$
Note that $\sum_{i=1}^mV_i\sim Binom\left(mn, \frac{q'}{n}\right)$,
implying that $E(\sum_{i=1}^mV_i)=mq'$. Hence, using Markov's
inequality we obtain
\begin{align}
\textmd{Pr}\left(\sum_{i=1}^m\textbf{I}_{\{V_i\geq 1\}}>k\right)\leq
\textmd{Pr}\left(\sum_{i=1}^mV_i>k\right)\leq\frac{mq'}{k}
\end{align}
Note that $q'\leq q$, where $q$ is the desired level of
$E(\mathcal{C}_\mathcal{S})$ and is typically less than 0.05.
Therefore, when $\frac{m}{k}$ is not much larger than 1, (say
$k=\frac{3m}{4}$ or $k=m-3$ where $m$ is large), $\frac{mq'}{k}$ is
very small and we can neglect it. Then, we obtain
\begin{align}
\frac{E\left(\sum_{i\in
\mathcal{S}(\textbf{P})}V_i\right)}{k}\approx
\frac{E(\sum_{i=1}^mV_i)}{k}=\frac{mq'}{k},
\end{align}
and the adjustment $q'=\frac{kq}{m}$ is necessary for assuring that
$E(\mathcal{C}_\mathcal{S})\leq q$.

The following procedure offers selection adjustment for any
selection rule. This procedure reduces to Procedure \ref{mainproc1}
when the selection rule is simple.
\begin{proc}[Selection-Adjusted Procedure]\label{mainproc2} $\,$
\smallskip

     \noindent 1. Apply the selection rule $\mathcal{S}$ to the ensemble of sets
$\textbf{P}$, identifying the selected set of families
$\mathcal{S}(\textbf{P})$. \\2. For each selected family $i$,
$i\in\mathcal{S}(\textbf{P})$, partition the ensemble of sets
\textbf{P} into $P_i$ (set of the p-values belonging to family $i$)
and $\textbf{P}^{(i)}$ (the ensemble of sets \textbf{P} without the
set $P_i$) and find:
\begin{align}
R_{min}(\textbf{P}^{(i)}):=\min_{\underline{p}}\{|\mathcal{S}(\textbf{P}^{(i)},
P_i=\underline{p})|:i\in\mathcal{S}(\textbf{P}^{(i)},
P_i=\underline{p})\}\label{rmin},
\end{align}
the minimal number of selected families when family $i$ is selected
and the p-values for other families do not change.
\\3. For each selected family $i$, apply $E(\mathcal{C})$-controlling procedure at level
\begin{align*}\frac{R_{min}(\textbf{P}^{(i)})q}{m}\end{align*}
\end{proc}
\begin{thm}\label{mainthm2} If the p-values across the families are independent,
then for any selection rule $\mathcal{S}(\textbf{P})$ and for any
$E(\mathcal{C})$-controlling procedure valid for the dependency
structure inside each family, the Selection-Adjusted Procedure
guarantees $E(\mathcal{C}_{\mathcal{S}})\leq q$.
\end{thm}
The proof of Theorem \ref{mainthm2} is given in Appendix B.
\begin{rem}
We could guarantee that in each selected family at least one
rejection is made by applying repeatedly the simple
Selection-Adjusted Procedure, selecting each time the families where
at least one rejection is made and adjusting the testing level at
each iteration according to the proportion of selected families (out
of all the families) at the previous iteration, until in each
selected family at least one rejection is made.

Interestingly, when each family consists only of one hypothesis and
each selected hypothesis is rejected at level $q$ if its p-value is
less than $q$, this iterative application of the Simple
Selection-Adjusted Procedure is equivalent to the Benjamini and
Hochberg procedure (BH hereafter, see Benjamini and Hochberg, 1995)
applied on the whole set of p-values (provided that at the first
iteration we select the hypotheses with p-values less than $q$).
Obviously, the testing procedure applied in each selected family is
an FWER-controlling procedure in this case, and the expected average
value of $\textbf{I}_{\{V \geq 1\}}$ over the selected families is
the expected proportion of type I errors out of all the selected
hypotheses. Since it is guaranteed that each selected hypothesis is
rejected, this is actually the FDR of the whole set of discoveries.
\end{rem}
\section{Selection of the families via multiple hypothesis
testing}\label{hyptest}
If the selected families are considered as scientific findings by
themselves, a situation often encountered in large testing problems,
it would be appropriate to address the erroneous selection of a
family, and control some error-rate of the selection process, as,
for example suggest Heller et al. (2009) for selecting gene sets in
microarray analysis and Sun and Wei (2011) for analyzing time-course
experiments. We may associate each family with its global null
(intersection) hypothesis and use the inside-family p-values in
order to construct a valid p-value for its intersection hypothesis.
(See Loughin (2004) for a systematic comparison of combining
functions that can be used for this purpose). Then we may apply a
multiple testing procedure on these combined p-values and select the
families for which the global null hypothesis is rejected. The
choice of the multiple testing procedure should be guided by the
error rate that we wish to control at the family level and the
dependency among the combined p-values.

Heller et al. (2009) address a similar problem of inference across
families of hypotheses in microarray analysis. They first select
promising gene sets and then look for differentially expressed genes
within these gene sets. They define an erroneous discovery of a set
if a set is selected while no gene in the set is differentially
expressed, or if a set is appropriately selected but one of the
genes in the set is erroneously discovered. They define the Overall
FDR criterion (OFDR), as the expected proportion of "erroneous"
discoveries of gene sets out of all the selected gene sets. This
error criterion is equivalent to $E(\mathcal{C}_\mathcal{S})$ for
$\mathcal{C}=\textbf{I}_{\{V>0\}}$ when it is guaranteed that in
each family (gene set) at least one rejection is made. This
condition is not always fulfilled. For example, when the signal in
the family is weak, it may be possible to see evidence that that
there is at least one signal in this family, but impossible to point
out where this signal is. In these cases our criterion does not
coincide with the OFDR. In order to see it, suppose an all-null
family is selected, and there are no rejections inside this family.
This family will have no contribution to $\mathcal{C}_\mathcal{S}$,
whereas it will have a contribution to the proportion of "erroneous"
discoveries of gene sets out of all the selected gene sets, as
defined by Heller et al. (2009).

In Heller et al. (2009) the division of the hypotheses into families
is determined by the problem. In many applications each hypothesis
carries two "tags", that is the hypotheses have two-ways structure.
The families can be constructed by pooling along either dimension.
In these cases the researcher should define the families by the most
important dimension for inference. In Section 4 we show an example
of such an application.



\subsection{Simple and non-simple selection rules}\label{simple}
In Section 2 we defined what a simple selection rule is. 
It is obvious that any single-step multiple testing procedure
satisfies this condition, since the cutoff for rejection does not
depend on the other p-values. In addition, any step-up and step-down
procedure defines a simple selection rule. See Appendix C for a
proof.

Important procedures which define selection rules that are not
simple are the adaptive FDR procedures (Benjamini and Hochberg,
2000, Storey et al., 2004, Benjamini et al., 2006, Blanchard and
Roquain, 2009). Let us consider the two-stage procedure given in
Benjamini et al. (2006). At the first stage, the hypotheses are
tested using the BH procedure at level $q'=\frac{q}{1+q}$. The
estimator of the number of true null hypotheses is
$\widehat{m}_0=m-R$, where $m$ is the total number of hypotheses and
$R$ is the number of rejected hypotheses at the first stage. If
$\widehat{m}_0=0$, the procedure rejects all the hypotheses.
Otherwise, the procedure rejects the hypotheses rejected by the BH
procedure at level $\frac{m}{\widehat{m}_0}q'$. The following shows
that this procedure is not simple.
\begin{ex}
Assume there are 3 hypotheses, $\{H_{0i}\}_{i=1}^3$. Let $P_1$,
$P_2$, and $P_3$ be the corresponding p-values. If
$P_1<\frac{q'}{3}$, $\frac{q'}{3}<P_2<\frac{2q'}{3}$, and
$\frac{3q'}{2}<P_3<3q'$, $\widehat{m}_0=1$ and all the hypotheses
are rejected. Fix $P_1$ and $P_3$, and increase $P_2$ so that
$\frac{2q'}{3}<P_2<q'$. Now $\widehat{m}_0=2$, therefore $H_{02}$ is
still rejected, but the total number of rejections changes from 3 to
2.
\end{ex}

\section{Associating SNPs with brain volume}\label{realdataex}
We would like to show the relevance of our approach to the voxelwise
genome-wide association study performed by Stein et al. (2010). The
authors explore the relation between each of 448293 Single
Nucleotide Polymorphisms (SNPs)  and each of 31622 voxels of the
entire brain across 740 elderly subjects, including subjects with
Alzeimer's disease, Mild Cognitive Impairment, and healthy elderly
controls from the Alzeimer's Disease Neuroimaging Initiative (ADNI).
The phenotype of interest was the percentage volume difference
relative to a sample specific template at each voxel, and a
regression was conducted at each SNP with the phenotype as the
dependent variable and the number of minor alleles, age and sex as
the independent variables (assuming the additive genetic model). In
the original analysis for each voxel only the most significantly
associated SNP was considered. Its p-value was "corrected" in order
to obtain uniform distribution when no SNP is associated with that
voxel. Then, the BH procedure was applied on the "corrected"
p-values. Two were found at the 0.5 level, but the 5 top SNPs were
selected for further research. This involved mapping the
significance of the voxels per each one of these 5 SNPs.

Note, that actually the authors first divided the set of hypotheses
into disjoint families, where each family was defined by a voxel,
and the hypotheses within the family where the hypotheses on the
association of each SNP with that voxel. The "corrected" p-value for
each voxel was the p-value for testing the global null hypothesis
for that voxel-family. Therefore, the authors selected the families,
i.e. the voxels where evidence for at least one non-null association
was obtained. Then, the authors considered the most associated SNPs
within the \emph{selected} voxels. So far, this analysis fits our
framework. At the last step, though, they returned and defined the 5
SNP-families as their findings looking at the significance of all
voxels within each SNP separately.

We would suggest another partition of the hypotheses into families.
As it can be understood from the paper, the authors are interested
to find SNPs associated with regions in the brain, and be able to
make maps of these regions. Therefore, it would be more appropriate
to define each family as the set of all the association hypotheses
for a specific SNP. This way, selection of the families would be
equivalent to the selection of SNPs, which could be followed by
finding the voxels associated with the selected SNPs.

The next question is what error-rates should be controlled in this
problem. It is obvious that the investigators do not wish to
emphasize each voxel-SNP pair where an association is found,
therefore there is no need to control for some error-rate globally,
on the combined set of all the discovered pairs. The emphasis is on
the selected SNPs and on the regions in the brain that could be
affected by these SNPs. Therefore, it would be reasonable to (1)
control for some
error-rate when selecting the SNPs 
(2) for each SNP, control for some error-rate when selecting the
voxels associated with that SNP,  and (3) control for the error-rate
in (2) on the average over the selected SNPs. The control on the
average over the selected guarantees the adjustment for selection
bias. The most common types of control in MRI analysis are FWER and
FDR. For FWER control on the average guarantees that the expected
proportion of SNPs where at least one voxel is erroneously declared
associated out of all the selected SNPs is bounded by a
pre-specified number (say 0.05). The FDR control on the average is a
more liberal property - it guarantees that the expected average over
the selected SNPs of the proportion of erroneously discovered voxels
per SNP is bounded.

 The control of some error-rate when selecting the SNPs could be achieved by defining the global null p-values
for each SNP and applying a multiple comparisons procedure on these
p-values, when the choice of the procedure should be guided by the
desired error-rate for the selection of SNPs (see Section
\ref{hyptest}). Theorems \ref{mainthm} and \ref{mainthm2} offer the
methods to obtain the control within SNPs and on the average over
the selected SNPs. According to these theorems, any commonly used
method in MRI research could be applied across voxels for each SNP
separately at the adjusted level: re-sampling or Random Field Theory
approaches for the control of FWER, or the BH procedure for the
control of FDR. Theorems \ref{mainthm} and \ref{mainthm2} however
assume independence across SNPs. This question is addressed
theoretically in the next section.
%
\section{Average control under dependency across the
families} All the results given so far hold when the p-values across
the families are independent. We will now consider the case where
the set of all the p-values possesses the positive regression
dependent on a subset (PRDS) property.

First recall that a set in $D$ in $R^n$ is increasing (decreasing)
if $x\in D$ and $y\geq x$ ($y\leq x$) implies that $y\in D$.

\begin{defn}(Benjamini and Yekutieli, 2001). The vector \textbf{X} is PRDS on $I_0$ if for any increasing set D
(where $x\in D$ and $y\geq x$ implies that $y\in D$) and for each
$i\in I_0$, $P(X \in D| X_i=x)$ is nondecreasing in $x$.
\end{defn}
In addition, we require that the selection rule be concordant, as
defined in Benjamini and Yekutieli (2005).
\begin{defn}(Benjamini and Yekutieli, 2005) A selection rule is
concordant if for each $i=1,\ldots,m$ and $k=1,\ldots,m$,
$\{\textbf{P}^{(i)}:k\leq R_{min}(\textbf{P}^{(i)})\}$  is a
decreasing set.
\end{defn}
It is easy to see that many selection rules are concordant. Both
selecting each family where its minimum p-value is less than $q$,
and selecting $k$ families with the smallest minimal p-values are
concordant selection rules. When the selection is made via
hypothesis testing, any step-up or step-down procedure is
concordant.
\begin{thm}\label{bhbonf} If the set of all the p-values is PRDS on the subset of
p-values corresponding to true null hypotheses, the selection rule
is concordant, and the procedure used for testing each selected
family is (1) Bonferroni procedure or (2) the BH procedure, then the
Selection-Adjusted Procedure guarantees in case (1):
\begin{align*}
E\Bigg(\frac{\sum_{i\in \mathcal{S}(\textbf{P})}V_i}{\max(R,
1)}\Bigg)\leq q
\end{align*}
and in case (2): \begin{align*}E\Bigg(\frac{\sum_{i\in
\mathcal{S}(\textbf{P})}Q_i}{\max(R, 1)}\Bigg)\leq q\end{align*}
\end{thm}
The proof is given in Appendix D.

\section{Discussion}
There have been very few works (outside Heller et al. (2009)
discussed in Section 3) that address formally the issue of inference
across families. We have mentioned Efron (2008) in the Introduction.
Other works dealing with this issue are Hu et al. (2010) and Sun and
Wei (2011). Neither of these last mentioned papers address the
testing of multiple families of hypotheses within the framework of
selective inference, which is the concern in our work. Testing each
family separately while attending to some error-rate control within
each tested family has an obvious advantage that the control is
achieved on the average across families. However, once only some
families are selected based on the same data, and inference is made
or reported only on the selected ones even this simple average
error-rate across families deteriorates. In this note we pointed at
this danger, formulated it, and offered simple - even if not optimal
- ways to address it. Sometimes, the situation faced calls for more
stringent control. This is the case when interest lies in assuring
simultaneous control of the error-rate across families, and not
merely on the average over the selected.

Such a concern for simultaneity of inference across selected
families can be formulated by $E(\max_{i=1,\ldots,m}
\mathcal{C}_i)$. For example, in the case
$\mathcal{C}_i=\mathbf{I}_{\{Q_i>\gamma\}}$ this is the probability
that in at least one family the false discovery proportion is
greater than $\gamma$. It is easy to see that controlling
$E(\mathcal{C}_i)$ at level $\frac{q}{m}$ in \textit{each} family
guarantees the control of this error criterion. However, usually in
applications the interest does not lie in all the families, but only
in the promising ones. Therefore we address only the selective goal
in this case.

It may sometimes happen that there are no rejections in a selected
family. For example, when the signal in the family is weak, it may
be possible to see evidence that there is at least one signal in
this family, but impossible to point out where the signal is. Some
investigators may claim that in this case the interpretation of the
results is not intuitive, therefore they wish to have at least one
rejection in each selected family. This can be easily done by
choosing appropriate selecting and testing procedures. It is easy to
see that if the testing procedure used for selecting the families is
a stepwise procedure with critical values less than or equal to
$\{\frac{iq}{m}\}_{i=1}^m$ and the global null p-value for each
family is not less than its minimal adjusted p-value (where the
adjustment is made according to the procedure used for testing the
selected families), it is guaranteed that in each selected family at
least one rejection is made.

In this paper we mainly addressed the goal of controlling some
error-rate within each family and on the average over the selected
families. Other types of error measures may be relevant as well. The
investigator might wish to control for some error-rate on the pooled
set of discoveries across all the families. This seems to be the
only concern in Efron (2008) and Hu et al. (2010). If the selected
families are considered as scientific findings by themselves, a
situation often encountered in large testing problems, it would be
appropriate to address the erroneous selection of a family, and
control some error-rate of the selection process, as, for example
suggest Heller et al. (2009) in the context of microarray analysis
and Sun and Wei (2011) in the context of time-course experiments.
The investigator may be interested in more than one type of error
measures. For example, one might wish to control for FDR within each
gene set, on the average over the selected gene sets and globally on
the pooled set of discovered genes across all the gene sets.
Therefore, an interesting research direction could be development of
the procedures controlling concurrently several error measures that
are of interest to the investigator.

\section{Appendix}
\subsection{Appendix A}
In Example 1, the formula for $E(\mathcal{C}_\mathcal{S})$  is the
following:
\begin{align}
\left(1-\left(1-\frac{q}{n}\right)^n\right)\frac{1-\left(1-q\right)^{nm}}{1-\left(1-q\right)^n}\label{toprove}
\end{align}
The proof is as follows. In this case
\begin{align*}
E(\mathcal{C}_\mathcal{S})=\sum_{i=1}^m\sum_{k=1}^m\frac{1}{k}\textmd{Pr}(V_i>0,
i\in \mathcal{S}(\textbf{P}), C_k^{(i)})
\end{align*}
Family $i$ is selected if its minimal p-value is less than $q$. Each
selected family is tested using Bonferroni procedure at level $q$.
Since all the null hypotheses are true, there is at least one type I
error in family $i$ if its minimal p-value is less than
$\frac{q}{n}$. Therefore, each family where at least one type I
error is made is selected. Now we obtain
\begin{align*}
E(\mathcal{C}_\mathcal{S})&=\sum_{i=1}^m\sum_{k=1}^m\frac{1}{k}\textmd{Pr}(V_i>0,
 C_k^{(i)})\\&=\sum_{i=1}^m\sum_{k=1}^m\frac{1}{k}\textmd{Pr}(V_i>0)
 \textmd{Pr}(C_k^{(i)})\\&=\sum_{i=1}^m\left(1-\left(1-\frac{q}{n}\right)^n\right)
\sum_{k=1}^m\frac{1}{k}\begin{pmatrix}
 m-1\\
k-1
\end{pmatrix}
\Bigg(1-\Big(1-q\Big)^n\Bigg)^{k-1}\Big(1-q\Big)^{n(m-k)}
\end{align*}
Let us define random variable $Y\sim Bin(m-1, 1-(1-q)^n)$. It is
easy to see that
\begin{align}E(\mathcal{C}_\mathcal{S})=m\left(1-\left(1-\frac{q}{n}\right)^n\right)E\left(\frac{1}{Y+1}\right).\label{fir}\end{align}
Using Lemma 1 in Benjamini et al. (2006) we obtain:
\begin{align}E\left(\frac{1}{Y+1}\right)=\frac{1-(1-q)^{nm}}{m(1-(1-q)^n)}.\label{sec}\end{align}
Substituting (\ref{sec}) in (\ref{fir}) we obtain the formula in
(\ref{toprove}).

\subsection{Appendix B}
\begin{center}\textit{Proof of Theorem \ref{mainthm2}}\end{center}
\begin{proof} For each error criterion
$E(\mathcal{C})$, let $\mathcal{C}_+$ be the support of random
variable $\mathcal{C}$.  As in Benjamini and Yekutieli (2005), we
define the following series of events:
\begin{align}
C_k^{(i)}:=\{\textbf{P}^{(i)}:\,R_{min}(\textbf{P}^{(i)})=k\}\label{cki}
\end{align}
According to the definition of $R_{min}(\textbf{P}^{(i)})$ (see
(\ref{rmin}) in Section 2), for each value of $\textbf{P}^{(i)}$ and
$P_i=\underline{p}$, such that $i\in \mathcal{S}(\textbf{P}^{(i)},
\underline{p})$,
 $R_{min}(\textbf{P}^{(i)})\leq |\mathcal{S}(\textbf{P}^{(i)},
\underline{p})|$. Therefore, 
\begin{align}
E(\mathcal{C}_\mathcal{S})&=E\Big(\frac{\sum_{i\in\mathcal{S}(\textbf{P})}\mathcal{C}_i}{\max(|\mathcal{S}(\textbf{P})|,
1)}\Big)\notag\\&\leq \sum_{i=1}^m\sum_{k=1}^m\frac{1}{k}\sum_{c\in
C_+}c\textmd{Pr}(\mathcal{C}_i=c, i\in \mathcal{S}(\textbf{P}),
C_k^{(i)})\label{rec}
\end{align}
 The expression in (\ref{rec}) is identical to the expression in
(\ref{forappb}) in the proof of Theorem \ref{mainthm}, but the
definition of $C_k^{(i)}$ here is different. In the proof of Theorem
\ref{mainthm} we use only the facts that the event $C_k^{(i)}$ is
defined on the space $P^{(i)}$ and that for each $i=1,\ldots,m$,
$\sum_{k=1}^m\textmd{Pr}(C_k^{(i)})=1$. These facts remain true for
the series of events defined in (\ref{cki}). Therefore, the
arguments used in the proof of Theorem \ref{mainthm} after obtaining
(\ref{forappb}) can be applied here.
\end{proof}
\subsection{Appendix C}
We will now prove that any step-up or step-down procedure defines a
simple selection rule. Let $\alpha_1, \alpha_2,\ldots, \alpha_m$ be
the critical values of the given procedure. Let $H_{0i}$ be a
certain rejected hypothesis , and $P_i$ be its p-value. We need to
show that when
all the p-values excluding $P_i$ are fixed and $P_i$ changes as long
as $H_{0i}$ is rejected, the total number of rejections remains
unchanged.

Assume this is a step-up procedure. Let $p_{(1)}^{(i)}\leq
\ldots\leq\ldots p_{(m-1)}^{(i)}$ be the ordered set of p-values
excluding $P_i$. If this is a step-up procedure , and
$p^{(i)}_{(k-1)}\leq \alpha_k, p^{(i)}_{(k)}>\alpha_{k+1}, \ldots,
p^{(i)}_{(m-1)}>\alpha_{m}$, the number of rejections is $k$ for any
value of $P_i$ which guarantees that $H_{0i}$ is rejected, i.e.
$P_i\leq \alpha_k$.

Now assume that this is a step-down procedure. Let $p_{(1)}\leq
p_{(2)}\leq\ldots\leq p_{(m)}$ be the ordered set of p-values.
Assume the number of rejections is $k$, thereby implying
$p_{(1)}\leq \alpha_1,\ldots, p_{(k)}\leq \alpha_k,
p_{(k+1)}>\alpha_{k+1}$. Since $H_{0i}$ is rejected, $P_i=p_{(j)}$
for some $j\leq k$. Let us fix all the p-values excluding $P_i$ and
change the value of $P_i$ so that $H_{0i}$ is still rejected. Assume
$\widetilde{p}_{(1)}\leq \widetilde{p}_{(2)}\leq\ldots\leq
\widetilde{p}_{(m)}$ is the ordered sequence of p-values after the
value of $P_i$ is changed. Now $P_i=\widetilde{p}_{(j')}$, and since
$H_{0i}$ is rejected, $\widetilde{p}_s\leq \alpha_s$ for each $s\leq
j'$. If $j'=j$, it is obvious that the number of rejections remains
unchanged. We will now deal separately with two cases: $j'<j$ and
$j'>j$.
\\(1) Assume $j'<j$. Then $j'\leq k$, therefore it remains
to show that $\widetilde{p}_{(s)}\leq \alpha_s$ for $s=j'+1,\ldots,
k$ and $\widetilde{p}_{(k+1)}>\alpha_{k+1}$. Note that
$\widetilde{p}_{(s)}=p_{(s-1)}\leq\alpha_{s-1}\leq\alpha_{s}$ for
$s=j'+1,\ldots,j$. For $s>j$, $\widetilde{p}_{(s)}=p_{(s)}$,
therefore now it is obvious that the number of rejections remains
unchanged.
\\(2) Assume $j'>j$. We will now show that $j'\leq k$.
Assume $j\leq k<j'$. Note that for each $j\leq s<j'$,
$\widetilde{p}_{(s)}=p_{(s+1)}$. Particularly,
$\widetilde{p}_k=p_{(k+1)}>\alpha_{k+1}>\alpha_k$, contradicting the
rejection of $H_{0i}$. After we have proved that $j'\leq k$, the
result follows immediately, since $\widetilde{p}_{(s)}=p_{(s)}$ for
$s>j'$.

\subsection{Appendix D}
\begin{center}\textit{Proof of Theorem \ref{bhbonf}}\end{center}
The proof uses the techniques developed in Benjamini and Yekutieli
(2001), (2005). The proof in case (2) is much more involved than in
case (1).
\subsubsection{Proof for case (1)}  
For each $i=1,\ldots,m$, let $m_i$ be the number of hypotheses in
family $i$ and $m_{0i}$ be the number of true null hypotheses in
family $i$. Let $H_{0ij}$ and $P_{ij}$, $j=1,\ldots,m_i$ be the
hypotheses and the p-values in family $i$, $i=1,\ldots,m$. We will
use the series of events $C_k^{(i)}$,
\begin{align*}
C_k^{(i)}:=\{\textbf{P}^{(i)}:\,R_{min}(\textbf{P}^{(i)})=k\}
\end{align*}
We will prove the following
\begin{align}
E\left(\frac{\sum_{i=1}^m
V_i}{\max(|\mathcal{S}(\textbf{P})|,1)}\right)&=\sum_{i=1}^m\sum_{k=1}^m\sum_{j=1}^{m_i}\frac{1}{k}\textmd{Pr}\left(i\in
\mathcal{S}(\textbf{P}), C_k^{(i)},
P_{ij}\leq\frac{kq}{mm_i}\right)\notag\\&\leq\sum_{i=1}^m\sum_{k=1}^m\sum_{j=1}^{m_{0i}}\frac{1}{k}\textmd{Pr}\left(
C_k^{(i)},
P_{ij}\leq\frac{kq}{mm_i}\right)\label{sp}\\&=\sum_{i=1}^m\sum_{k=1}^m\sum_{j=1}^{m_i}\frac{1}{k}\textmd{Pr}\left(
C_k^{(i)}|P_{ij}\leq\frac{kq}{mm_i}\right)Pr\left(P_{ij}\leq\frac{kq}{mm_i}\right)\notag\\&\leq\sum_{i=1}^m\sum_{k=1}^m\sum_{j=1}^{m_{0i}}\frac{1}{k}\textmd{Pr}\left(
C_k^{(i)}|P_{ij}\leq\frac{kq}{mm_i}\right)\frac{kq}{mm_i}\label{cont}\\&=\frac{q}{m}\sum_{i=1}^m\frac{1}{m_{i}}\sum_{j=1}^{m_{0i}}\sum_{k=1}^m\textmd{Pr}\left(
C_k^{(i)}|P_{ij}\leq\frac{kq}{mm_i}\right)\label{lastbonf}
\end{align}

Inequality in (\ref{sp}) is obtained by dropping the condition $i\in
\mathcal{S}(\textbf{P})$. Inequality in (\ref{cont}) is true since
the p-values corresponding to true null hypotheses have a uniform
(or stochastically larger) distribution. 
We will now prove that for any $i=1,\ldots,m$, and $j=1,\ldots,m_0$
\begin{align}
\sum_{k=1}^m\textmd{Pr}\left(
C_k^{(i)}|P_{ij}\leq\frac{kq}{mm_i}\right)\leq 1.
\end{align}
Since the selection rule is concordant, the set
$D_k^{(i)}=\cup_{j=1}^kC_j^{(i)}$, which can be written as
$\{P^{(i)}:R_{min}(P^{(i)})<k+1\}$, is an increasing set. The PRDS
property on the subset of the p-values corresponding to the true
null hypotheses implies that for any $i=1,\ldots,m$,
$j=1\ldots,m_{0i}$:
\begin{align*}
\textmd{Pr}\left(D_k^{(i)}|P_{ij}\leq \alpha\right)\leq
\textmd{Pr}\left(D_k^{(i)}|P_{ij}\leq \alpha'\right)
\end{align*}
for any $\alpha\leq \alpha'$. Now, we obtain for any
$k=1,\ldots,m-1$:
\begin{align}
&\textmd{Pr}\left(D_k^{(i)}|P_{ij}\leq\frac{kq}{mm_i}\right)+\textmd{Pr}\left(C_{k+1}^{(i)}|P_{ij}\leq\frac{(k+1)q}{mm_i}\right)\notag\\&\leq
\textmd{Pr}\left(D_k^{(i)}|P_{ij}\leq\frac{(k+1)q}{mm_i}\right)+\textmd{Pr}\left(C_{k+1}^{(i)}|P_{ij}\leq\frac{(k+1)q}{mm_i}\right)\notag\\&=
\textmd{Pr}\left(D_{k+1}^{(i)}|P_{ij}\leq\frac{(k+1)q}{mm_i}\right)\label{bonfprds}.
\end{align}
Applying repeatedly inequality (\ref{bonfprds}) for
$k=1,\ldots,m-1$, and using the fact that $C_1^{(i)}=D_1^{(i)}$ we obtain
\begin{align}
\sum_{k=1}^m\textmd{Pr}\left(C_k^{(i)}\,|\,P_{ij}\leq\frac{kq}{mm_i}\right)&\leq
\textmd{Pr}\left(D_m^{(i)}|P_{ij}\leq\frac{mq}{mm_i}\right)\leq
1\label{mainbnf}
\end{align}
Using (\ref{lastbonf}) and (\ref{mainbnf}) we obtain
\begin{align*}
E\left(\frac{\sum_{i=1}^m
V_i}{\max(|\mathcal{S}(\textbf{P})|,1)}\right)&\leq\frac{q}{m}\sum_{i=1}^m\left(\frac{m_{0i}}{m_{i}}\right)\leq
q.
\end{align*}
%
%
\subsubsection{Proof for case (2)} Let $P_i$ be the set of
p-values corresponding to family $i$. For each $j=1,\ldots, m_{0i}$,
let $P_i^{(ij)}$ denote the set of the remaining $m_i-1$ p-values
after dropping $P_{ij}$. Let us define the following series of
events on the range of $P_i^{(ij)}$. For family $i$, let
$B_{r_i}^{(ij)}[k]$ denote the event in which if $H_{0ij}$ is
rejected by the BH procedure at level $\frac{kq}{m}$, $r_i$
hypotheses (including $H_{0ij}$) are rejected alongside with it. We
will use again the series of events $C_k^{(i)}$, defined in
(\ref{cki}) in Appendix B. Using the fact that the p-values
corresponding to the true null hypotheses have a uniform (or
stochastically larger) distribution, we obtain
\begin{align*}
E\left(\frac{\sum_{i\in \mathcal{S}}
Q_i}{\max(|\mathcal{S}(\textbf{P})|,1)}\right)&=\sum_{i=1}^m\sum_{k=1}^m\frac{1}{k}\sum_{r_{i}=1}^{m_i}\frac{1}{r_i}\sum_{j=1}^{m_{0i}}\textmd{Pr}\left(i\in
\mathcal{S}(\textbf{P}), C_k^{(i)}, B_{r_i}^{(ij)}[k],
P_{ij}\leq\frac{r_ikq}{mm_i}\right)\\&\leq\sum_{i=1}^m\sum_{k=1}^m\frac{1}{k}\sum_{r_{i}=1}^{m_i}\frac{1}{r_i}\sum_{j=1}^{m_{0i}}\textmd{Pr}\left(
 C_k^{(i)}, P_{ij}\leq\frac{r_ikq}{mm_i},
B_{r_i}^{(ij)}[k]\right)\\&\leq\sum_{i=1}^m\sum_{k=1}^m\frac{1}{k}\sum_{r_{i}=1}^{m_i}\frac{1}{r_i}\sum_{j=1}^{m_{0i}}\frac{r_ikq}{mm_i}\cdot\textmd{Pr}\left(C_k^{(i)},
B_{r_i}^{(ij)}[k]\,|\,P_{ij}\leq\frac{r_ikq}{mm_i}\right)\\&=\frac{q}{m}\sum_{i=1}^m\frac{1}{m_i}\sum_{j=1}^{m_{0i}}\sum_{k=1}^m\sum_{r_{i}=1}^{m_i}\textmd{Pr}\left(C_k^{(i)},
B_{r_i}^{(ij)}[k]\,|\,P_{ij}\leq\frac{r_ikq}{mm_i}\right)
\end{align*}
Now the fact that $m_{0i}\leq m_i$ reduces case (2) to the
inequality
\begin{align}
\sum_{k=1}^m\sum_{r_{i}=1}^{m_i}\textmd{Pr}\left(C_k^{(i)},
B_{r_i}^{(ij)}[k]\,|\,P_{ij}\leq\frac{r_ikq}{mm_i}\right)\leq1\label{leq1}
\end{align}
where $i=1,\ldots,m,$ $j=1,\ldots,m_{0i}$.

For each $i=1,\ldots,m$, $t=1,\ldots m m_i$, let us define the group
$$I_t=\{(a, b):a\in\{1,\ldots,m\}, b\in\{1,\ldots,m_i\},\, ab=t\}.$$
Obviously, $I_t$ is a finite set. Note that
\begin{align*}
\sum_{k=1}^m\sum_{r_{i}=1}^{m_i}\textmd{Pr}\left(C_k^{(i)},
B_{r_i}^{(ij)}[k]\,|\,P_{ij}\leq\frac{r_ikq}{mm_i}\right)=\\\\\sum_{t=1}^{mm_i}\sum_{(k,
r_i)\in I_t}\textmd{Pr}\left(C_k^{(i)},
B_{r_i}^{(ij)}[k]\,|\,P_{ij}\leq\frac{tq}{mm_i}\right)
\end{align*}
Therefore, (\ref{leq1}) can be written in the form
\begin{align}
\sum_{t=1}^{mm_i}\sum_{(k, r_i)\in I_t}\textmd{Pr}\left(C_k^{(i)},
B_{r_i}^{(ij)}[k]\,|\,P_{ij}\leq\frac{tq}{mm_i}\right)\leq
1\label{leq1t}
\end{align}
For each family $i$ and its hypothesis $H_{0ij}$, $i=1,\ldots,m$,
$j=1,\ldots,m_{0i}$, let us define
\begin{align*}
A_s^{(ij)}\triangleq\bigcup_{t=1}^s\bigcup_{(k,r_i)\in
I_t}\left(C_k^{(i)}\bigcap B_{r_i}^{(ij)}[k]\right).
\end{align*}
The key statement to prove (\ref{leq1t}) is the following
proposition.
\begin{prop}\label{propseq}
For any $s=1,\ldots, mm_i-1$, $i=1,\ldots,m$,  $j=1,\ldots,m_{0i}$,
one has the inequality
\begin{align}
\textmd{Pr}\left(A_s^{(ij)}\,|\,P_{ij}\leq\frac{sq}{mm_i}\right)&+\sum_{(k,
r_i)\in I_{s+1}}\textmd{Pr}\left(C_k^{(i)},
B_{r_i}^{(ij)}[k]\,|\,P_{ij}\leq\frac{(s+1)q}{mm_i}\right)\notag\\&\leq
\textmd{Pr}\left(A_{s+1}^{(ij)}\,|\,P_{ij}\leq\frac{(s+1)q}{mm_i}\right)\label{iter}
\end{align}
\end{prop}
Let us show that this proposition implies (\ref{leq1t}). Note that
$$\textmd{\textmd{Pr}}\left(A_1^{(ij)}|P_{ij}\leq\frac{q}{mm_i}\right)=\sum_{(k,
r_i)\in I_{1}}\textmd{Pr}\left(C_k^{(i)},
B_{r_i}^{(ij)}[k]\,|\,P_{ij}\leq\frac{q}{mm_i}\right).$$ Now, a
consequent application of (\ref{iter}) with $s=1,\ldots,mm_i-1$
leads to the inequality
\begin{align*}
\sum_{t=1}^{mm_i}\sum_{(k, r_i)\in I_t}\textmd{Pr}\left(C_k^{(i)},
B_{r_i}^{(ij)}[k]\,|\,P_{ij}\leq\frac{tq}{mm_i}\right)\leq\textmd{Pr}
\left(A_{mm_i}^{(ij)}\,|\,P_{ij}\leq q\right)\leq 1,
\end{align*}
which implies (\ref{leq1t}).

\emph{Proof of Proposition \ref{propseq}.} Let us show that
$A_s^{(ij)}$ is an increasing set. Let
$$D_k^{(i)}\triangleq\{P^{(i)}:R_{min}(P^{(i)})<k+1\}$$ and
$$G_{r_i}^{(ij)}[k]\triangleq\{P_i^{(ij)}:
P_{(r_i)}^{(ij)}>\frac{(r_i+1)kq}{mm_i},
P_{(r_i+1)}^{(ij)}>\frac{(r_i+2)kq}{mm_i},\ldots,P_{(m_i-1)}^{(ij)}>\frac{kq}{m}\},$$
where $\{P_{i(1)}^{(ij)}\leq P_{i(2)}^{(ij)}\leq \ldots\leq
P_{i(m-1)}^{(ij)}\}$ is the ordered set of p-values in the range of
$P_i^{(ij)}$. It is easy to see that
\begin{align*}
A_s^{(ij)}=\bigcup_{t=1}^s\bigcup_{(k,r_i)\in
I_t}\left(D_k^{(i)}\bigcap G_{r_i}^{(ij)}[k]\right)
\end{align*}
Obviously, both $D_k^{(i)}$ and $G_{r_i}^{(ij)}[k]$ are increasing
sets. Unions and intersections of increasing sets are also an
increasing set, hence $A_s^{(ij)}$ is an increasing set. The PRDS
property on the subset of p-values corresponding to the true null
hypotheses implies that for each $i=1\ldots,m$, $j=1,\ldots,m_{0i},
$ and any $\alpha\leq\alpha'$
\begin{align}\label{prds}
\textmd{Pr}\left(A_s^{(ij)}|P_{ij}\leq \alpha\right)\leq
\textmd{Pr}\left(A_s^{(ij)}|P_{ij}\leq \alpha'\right)
\end{align}
It is obvious that for $(k,r)\neq(k', r')$, the sets
$\left(C_k^{(i)}\bigcap B_{r_i}^{(ij)}[k]\right)$ and
$\left(C_{k'}^{(i)}\bigcap B_{{r'}_i}^{(ij)}[k']\right)$ are
disjoint. Therefore for $t\neq t'$, $\bigcup_{(k,r_i)\in I_t}
\left(C_k^{(i)}\bigcap B_{r_i}^{(ij)}[k]\right)$ and
$\bigcup_{(k,r_i)\in I_{t'}}\left(C_k^{(i)}\bigcap
B_{r_i}^{(ij)}[k]\right)$ are disjoint as well. Now we obtain for
any $\alpha$
\begin{align}\label{comb}
\textmd{Pr}\left(A_s^{(ij)}\,|\,P_{ij}\leq
\alpha\right)=\sum_{t=1}^{s}\sum_{(k, r_i)\in
I_t}\textmd{Pr}\left(C_k^{(i)}, B_{r_i}^{(ij)}[k]\,|\,P_{ij}\leq
\alpha\right)
\end{align}
Using (\ref{prds})  and (\ref{comb}) we obtain that for any
$s=1,\ldots,mm_i-1$:
\begin{align*}
\textmd{Pr}\left(A_s^{(ij)}|P_{ij}\leq\frac{sq}{mm_i}\right)+\sum_{(k,
r_i)\in I_{s+1}}\textmd{Pr}\left(C_k^{(i)},
B_{r_i}^{(ij)}[k]\,|\,P_{ij}\leq\frac{(s+1)q}{mm_i}\right)\leq\notag\\
\textmd{Pr}\left(A_s^{(ij)}|P_{ij}\leq\frac{(s+1)q}{mm_i}\right)+\sum_{(k,
r_i)\in I_{s+1}}\textmd{Pr}\left(C_k^{(i)},
B_{r_i}^{(ij)}[k]\,|\,P_{ij}\leq\frac{(s+1)q}{mm_i}\right)=\notag\\
\textmd{Pr}\left(A_{s+1}^{(ij)}|P_{ij}\leq\frac{(s+1)q}{mm_i}\right)
\end{align*}
and Proposition \ref{propseq} follows.

\section{References} \textsc{Benjamini, Y. and Heller, R.}
(2007) False discovery rate for spatial data. \textit{Journal of
American Statistical Association,} \textbf{102}, 1272-1281.
\\\textsc{Benjamini,
Y. and Hochberg, Y.} (1995) Controlling the false discovery rate - A
practical and powerful approach to multiple testing. \textit{Journal
of the Royal Statistical Society}, \textit{Series B (Metallurgy)},
\textbf{57}, 289-300.
\\\textsc{Benjamini, Y. and Hochberg, Y.} (2000)  On the adaptive control of the false discovery fate in multiple testing with independent
statistics. \textit{Journal of educational and behavioral
statistics,} \textbf{25}(1), 60-83.
\\\textsc{Benjamini, Y., Krieger, A. M., and Yekutieli, D.} (2006) Adaptive
linear step-up false discovery rate controlling procedures.
\textit{Biometrika}, \textbf{93} (3):491-507.
\\\textsc{Benjamini, Y. and Yekutieli, D.} (2001) The
control of the false discovery rate in multiple testing under
dependency. \textit{Annals of Statistics,} \textbf{29}, 1165-1188.
\\\textsc{Benjamini, Y. and Yekutieli, D.} (2005) False discovery
rate-adjusted multiple confidence intervals for selected parameters.
\textit{Journal of the American Statistical Association,}
\textbf{100}, 71-93.
\\\textsc{Blanchard, G. and Roquain, E.} (2009) Adaptive False Discovery Rate
Control under Independence and Dependence. \textit{Journal of
machine learning research}, \textbf{10}, 2837-2871.
\\\textsc{Efron, B.} (2008) Simultaneous inference: when should
hypotheses testing problems be combined? \textit{Annals of Applied
Statistics}, \textbf{2}: 197-223.
\\\textsc{Efron, B. and Tibshirani, R.} (2002) Empirical Bayes
methods and false discovery rates for microarrays. \textit{Genetics
Epidemiology,} \textbf{23}, 70–86.
\\\textsc{Farcomeni, A.} (2008) A review of modern multiple hypothesis testing, with
particular attention to the false discovery proportion.
\textit{Statistical Methods in Medical Research,} \textbf{17},
347-388.
\\\textsc{Genovese, C.R. and Wasserman L.}
(2006) Exceedance control of the false discovery proportion.
\textit{Journal of the American Statistical Association,}
\textbf{101}, 1408-1417.
\\\textsc{Heller, R., Manduchi, E., Grant G.R. and Ewens, W.J.} (2009). A flexible
two-stage procedure for identifying gene sets that are
differentially expressed. \textit{Bioinformatics,} \textbf{25},
1019-1205.
\\\textsc{Hu J.X., Zhao, H.Y. and Zhou, H.H.} (2010) False Discovery rate
control with groups. \textit{Journal of the American Statistical
Association}, \textbf{105}, 1215-1227.
\\\textsc{Lehmann, E.L. and Romano J.P.}
(2005) Generalizations of the familywise error rate. \textit{Annals
of Statistics,} \textbf{33}, 1138-1154.
\\\textsc{Loughin, T.} (2004) A systematic comparison of methods for
combining p-values from independent tests. \textit{Computational
statistics and Data Analysis,} \textbf{47}, 467-485.
\\\textsc{Pacifico, M. P.,
Genovese, C., Verdinelli, I. and Wasserman, L.} (2004) False
discovery control for random fields. \textit{J. Multivariate Anal.,}
\textbf{98}, 1441-1469.
\\\textsc{Sarkar, S.K.} (2007) Step-up procedures controlling
generalized FWER and generalized FDR. \textit{Annals of Statistics},
\textbf{35}(6), 2405-2420.
\\\textsc{Stein, J.L., Hua, X., Lee S., Ho A.J., Leow A.D., Toga, A.W., Saykin, A.J., Shen L.,
Foroud T., Pankratz N., Huentelman M.J., Craig, D.W.,Gerber, J.D.,
Allen, A.N., Corneveaux, J.J., DeChairo B.M., Potkin, S.G.  Weiner
M.W. and Thompson, P.M.} (2010) Voxelwise genome-wide association
study (vGWAS). \textit{NeuroImage}, \textbf{53}, 1160-1174.
\\\textsc{Storey, J. D.} (2003) The positive false
discovery rate: A Bayesian interpretation and the q-value.
\textit{Annals of Statistics}, \textbf{31}, 2013-2035.
\\\textsc{Storey, J., Taylor, J., and Siegmund, D.} (2004) Strong control,
conservative point estimation, and simultaneous conservative
consistency of false discovery rates: A unified approach.
\textit{Journal of the Royal Statistical Society, Series B},
\textbf{66} : 187-205.
\\\textsc{Subramanian, A., Tamayo, P.,
Mootha, V., Mukherjee, S., Ebert, B., Gillette, M., Paulovich, A.,
Pomeroy, S., Golub, T., Lander, E., and Mesirov, J.} (2005) Gene set
enrichment analysis: A knowledge-based approach for interpreting
genome-wide expression profiles. \textit{Proceedings of the national
academy of sciences of the USA}, \textbf{102}(43) : 15545-15550.
\\\textsc{Sun, W. and Wei, Z.} (2011) Multiple testing for pattern
identification, with application to microarray time-course
experiments. \textit{Submitted for publication.}
\\\textsc{van der Laan, M.J., Dudoit, S. and
Pollard, K.S.}, (2004) Augmentation procedures for control of the
generalized family-wise error rate and tail probabilities for the
proportion of false positives. \textit{Statistical Applications in
Genetics and Molecular Biology}, \textbf{3}(1).
\\\textsc{Williams, V.S.L., Jones, L.V. and Tukey J.W.} (1999) Controlling error
in multiple comparisons, with examples from state-to-state
differences in educational achievement.  \textit{Journal of
Educational and Behavioral Statistics,} \textbf{24}, 42-69.
\end{document}